\documentclass[a4paper,10pt]{article}
\usepackage{times}
\usepackage[english]{babel}
\usepackage{amssymb,amsmath}
\numberwithin{equation}{section}
\usepackage{amsthm}
\usepackage{array}
\usepackage{wasysym}
\usepackage{hyperref}
\usepackage[height=22cm , width = 16cm , top = 4cm , left = 3cm, a4paper]{geometry}
\usepackage{amssymb}
\usepackage{amsmath}
\usepackage{epsfig}
\usepackage{verbatim}
\usepackage{multicol}
\usepackage{graphicx, color, psfrag}
\usepackage[numbers,sort]{natbib}
 \newtheorem{theorem}{Theorem}[section]
\newtheorem{remark}[theorem]{Remark}

\theoremstyle{definition}
\newtheorem{definition}{Definition}[section]

\begin{document}

\thispagestyle{empty}

\author{{Beih S. El-Desouky,$^{}$ $$\footnote{Corresponding author: b\_desouky@yahoo.com } { } Rabab S. Gomaa and Alia M. Magar
\newline{\it{{}  }}
 } { }\vspace{.2cm}\\
  \it Department of Mathematics, Faculty of Science, Mansoura University, \\35516 Mansoura, Egypt \vspace{.2cm} \\
{ }
}
\title{The 2-Variable Unified Family of Generalized Apostol-Euler, Bernoulli and Genocchi Polynomials}

\maketitle

\hrule \vskip 8pt

\begin{quote}
{\small {\em\bf Abstract}\\{In this paper, we introduce The 2-variable unified family of generalized Apostol-Euler, Bernoulli and Genocchi polynomials and derive some implicit summation formulae and general
symmetry identities. The result extend some known
summations and identities of generalized Bernoulli, Euler and
Genocchi numbers and polynomials.}}
\end{quote}


\section{Introduction}
 The 2-variable general polynomials (2VGP) $p_{n}(x, y)$ are defined by means of the
following generating function \cite{Khan}:
\begin{equation}\label{1}
  e^{xt}\; \varphi(y,t)=\sum_{n=0}^{\infty}p_{n}(x, y)\frac{t^{n}}{n!},\;p_{0}(x, y)=1,
\end{equation}
where $\varphi(y, t)$ has (at least the formal) series expansion
\begin{equation}\label{2}
  \varphi(y,t)=\sum_{n=0}^{\infty}\varphi_{n}(y)\frac{t^{n}}{n!},\;\varphi_{0}(y)\neq0.
\end{equation}
 The 2-variable general polynomials $p_{n}(x, y)$ contains a number of important special polynomials of two variables.\\
Generating functions for certain members belonging to the (2VGP) are given as following:\\
The higher order Hermite polynomials, sometimes called the Kampé de Feriet polynomials of order m or the Gould–
Hopper polynomials $H_{n}^{(m)}(x, y)$ defined by the generating function\cite{Gould}
\begin{equation}\label{3}
  e^{xt+yt^{m}}=\sum_{n=0}^{\infty}{H}_{n}^{(m)}(x,y)\frac{t^{n}}{n!}.
\end{equation}
The 2-variable Hermite Kampé de Feriet polynomials $H_{n}(x, y)$ defined
by the generating function \cite{Appell}
\begin{equation}\label{4}
  e^{xt+yt^{2}}=\sum_{n=0}^{\infty}{H}_{n}(x,y)\frac{t^{n}}{n!}.
\end{equation}
The 2-variable generalized Laguerre polynomials $_{m}L_{n}(y, x)$ are defined by the following generating function \cite{Dattoli2}
\begin{equation}\label{5}
 e^{xt}C_{0}(-yt^{m})=\sum_{n=0}^{\infty}  {}_{m}L_{n}(y, x)  \frac{t^{n}}{n!},
\end{equation}
where $C_{0}(y)$ is the 0-th order Tricomi function\cite{Andrews}
\begin{equation}\label{6}
  C_{0}(y)=\sum_{r=0}^{\infty}\frac{(-1)^{r} x^{r}}{(r!)^{2}}.
\end{equation}
The 2-variable Laguerre polynomials $L_{n}(y, x)$ are defined by the following generating function \cite{Dattoli1}
\begin{equation}\label{7}
 e^{xt}C_{0}(yt)=\sum_{n=0}^{\infty}  L_{n}(y, x)  \frac{t^{n}}{n!}.
\end{equation}
The 2-variable truncated exponential polynomials of order r \;\;$e_{n}^{(r)}(x, y)$ are defined by the generating function \cite{Dattoli4}
\begin{equation}\label{9}
  \frac{e^{xt}}{(1-yt^{r})}=\sum^{\infty}_{n=0} e_{n}^{(r)}(x,y)\;\frac{t^{n}}{n!}.
\end{equation}
In particular, we note that
\begin{equation*}
   e_{n}^{(2)}(x,y)=n!\;\;{}_{[2]}e_{n}(x,y),
\end{equation*}
\begin{equation*}
{}_{[2]}e_{n}(x,1)={}_{[2]}e_{n}(x),
\end{equation*}
where ${}_{[2]}e_{n}(x,y)$ denotes the 2-variable truncated exponential polynomials \cite{Dattoli3}.\\
The 2-variable truncated exponential polynomials \;\;${}_{[2]}e_{n}^{(r)}(x, y)$ are defined by the generating function \cite{Dattoli3}
\begin{equation}\label{10}
  \frac{e^{xt}}{(1-yt^{2})}=\sum^{\infty}_{n=0} {}_{[2]}e_{n}(x,y)\;\frac{t^{n}}{n!}.
\end{equation}
The generalized Apostol-Bernoulli polynomials $ B_{n}^{(\alpha)}(x;\lambda)$ of order $\alpha\in \mathbb{C},$ the generalized Apostol-Euler polynomials
$ E_{n}^{(\alpha)}(x;\lambda)$ of order $\alpha\in \mathbb{C}$ and the generalized Apostol-Genocchi polynomials
$ G_{n}^{(\alpha)}(x;\lambda)$ of order $\alpha\in \mathbb{C}$ are defined, see (\cite{lu1}, \cite{lu2}, \cite{lu3} and \cite{lu5}) respectively, through the generating function by:
\begin{equation*}
 \left( \frac{t}{\lambda e^{t}-1}\right)^{\alpha}\;e^{xt}=
 \sum_{n=0}^{\infty}B_{n}^{^{(\alpha)}}(x;\lambda)\frac{t^{n}}{n!},
\end{equation*}
\begin{equation}\label{20}
 (\mid t\mid <2\pi,\;\; \text{when}\;\; \lambda=1;\;\; \mid t\mid <\mid\log \lambda\mid,\;\; \text{when}\;\; \lambda\neq 1,\;\;1^{\alpha}:=1)
\end{equation}
\begin{equation*}
 \left( \frac{2}{\lambda e^{t}+1}\right)^{\alpha}\;e^{xt}=
 \sum_{n=0}^{\infty}E_{n}^{^{(\alpha)}}(x;\lambda)\frac{t^{n}}{n!},
\end{equation*}
\begin{equation}\label{21}
 (\mid t\mid <\pi,\;\; \text{when}\;\; \lambda=1;\;\; \mid t\mid <\mid\log (-\lambda)\mid,\;\; \text{when}\;\; \lambda\neq 1,\;\;1^{\alpha}:=1)
\end{equation}
and
\begin{equation*}
 \left( \frac{2t}{\lambda e^{t}+1}\right)^{\alpha}\;e^{xt}=
 \sum_{n=0}^{\infty}G_{n}^{^{(\alpha)}}(x;\lambda)\frac{t^{n}}{n!},
\end{equation*}
\begin{equation}\label{22}
 (\mid t\mid <\pi,\;\; \text{when}\;\; \lambda=1;\;\; \mid t\mid <\mid\log(- \lambda)\mid,\;\; \text{when}\;\; \lambda\neq 1,\;\;1^{\alpha}:=1)
\end{equation}
Khan at el. \cite{Khan2} introduced the 2-variable Apostol type polynomials of order $\alpha,$ by means of the following generating function:
\begin{equation}\label{23}
  \sum_{n=0}^{\infty}{}_{p}F_{n}(x,y;\lambda;\mu,\nu)\frac{t^{n}}{n!}=(\frac{2^{\mu}t^{\nu}}{\lambda e^{t}+1})^{\alpha}e^{xt}\varphi(y,t),\;\;\; |t|<|\log(-\lambda)|.
\end{equation}

In this article, we introduce the 2-variable unified family of generalized Apostol-Euler, Bernoulli and Genocchi polynomials
and investigate their properties. Some important summation formulas are
given. Moreover, general symmetric identity are derived .

\section{The 2-variable unified family of generalized Apostol-Euler, Bernoulli and Genocchi polynomials}
The 2-variable unified family of generalized Apostol-Euler, Bernoulli and Genocchi polynomials of order $r,$ denoted by $ {}_{p}\mathbb{M}_{n}^{^{(r)}}(x,k,a,b;\bar\alpha_{r})$ will be defined as the discrete Apostol type convolution of the 2-variable general polynomials $ p_{n}(x, y)$.
\begin{definition}
Let $a, b$ and $c$ be positive integers with the condition $a\neq b$. A new generalization of the Apostol Hermite-Genocchi polynomials
$\mathbb{}_{H}{M}_{n}^{^{(r)}}(x,y;a,b,c;\bar\alpha_{r})$ for nonnegative integer $n$ is defined by means by the generating function
\begin{equation}\label{n}
\sum_{n=0}^{\infty}\mathbb{}_{H}{M}_{n}^{^{(r)}}(x,y;k,a,b;\bar\alpha_{r})\frac{t^{n}}{n!}=\frac{(-1)^{r}t^{rk}2^{r(1-k)}}{\prod\limits^{r-1}_{i=0}(\alpha_{i}b^{t}-a^{t})} \,\,c^{xt}\varphi (y,t),
\end{equation}
\end{definition}

Setting $ c=e$  and $\varphi (y,t)=1 $ in \eqref{n}, we get the following definition.
\begin{definition}
A unified family $ \mathbb{M}_{n}^{(r)}(x;a,b;\bar\alpha_{r})$ of generalized Apostol-Euler,Bernoulli and Genocchi polynomials is given by means of the generating function
\begin{equation*}
\sum_{n=0}^{\infty}\mathbb{M}_{n}^{^{(r)}}(x;k,a,b;\bar\alpha_{r})\frac{t^{n}}{n!}=\frac{(-1)^{r}t^{rk}2^{r(1-k)}}{\prod\limits^{r-1}_{i=0}(\alpha_{i}b^{t}-a^{t})} \,\,e^{xt},
\end{equation*}
\begin{equation}\label{11}
 \left( \mid t\mid <\mid \frac{\log(\alpha_{i})}{\log(\frac{b}{a})}\mid;\,\, a,b\in \mathbb{R}^{+};\,\, \alpha_{i}\neq1;\,\,for all\,\, i=0,1,\cdots,r-1\right),
\end{equation}
where $r\in\mathbb{C};\,\, \bar\alpha_{r}=(\alpha_{0},\alpha_{1},\cdots,\alpha_{r-1})$ is a sequence of complex numbers.
\end{definition}
\begin{remark}
If we set $x=0$ in \eqref{11}, then we obtain the new unified family of generalized Apostol-Euler, Bernoulli and Genocchi numbers, defined as\\
\begin{equation}\label{12}
  \sum_{n=0}^{\infty}\mathbb{M}_{n}^{^{(r)}}(k,a,b;\bar\alpha_{r})\frac{t^{n}}{n!}=\frac{(-1)^{r}t^{rk}2^{r(1-k)}}{\prod\limits^{r-1}_{i=0}(\alpha_{i}b^{t}-a^{t})}.
\end{equation}
\end{remark}

Setting $ c=e$ in \eqref{n}, we get the following definition.
\begin{definition}
The 2-variable unified family of generalized Apostol-Euler, Bernoulli and Genocchi polynomials of order $(r)$ ${}_{p}\mathbb{M}_{n}^{^{(r)}}(x,,y;k,a,b;\bar\alpha_{r})$ is defined by the following generating function
\begin{equation}\label{13}
\sum_{n=0}^{\infty}{}_{p}\mathbb{M}_{n}^{^{(r)}}(x,y;k,a,b;\bar\alpha_{r})\frac{t^{n}}{n!}=\frac{(-1)^{r}t^{rk}2^{r(1-k)}}{\prod\limits^{r-1}_{i=0}(\alpha_{i}b^{t}-a^{t})} \,\,e^{xt}\varphi (y,t),
\end{equation}
\end{definition}

\subsection{Special cases}
The generating function in (\eqref{13})gives types of polynomials as special cases, for example
\begin{enumerate}
  \item setting $\alpha_{i}=-\lambda,\;b=e\; and\;a=1 \; in \;\eqref{13},$ we have
\begin{eqnarray*}
 \sum_{n=0}^{\infty}{}_{p}\mathbb{M}_{n}^{^{(r)}}(x,y;k,1;-\lambda)\frac{t^{n}}{n!}&=& \frac{(-1)^{r} 2^{r(1-k)}t^{tk}}{(-\lambda e^{t}-1)^{r}}e^{xt}\;\;\varphi(y,t)  \\
&=& \left(\frac{2^{1-k}t^{k}}{\lambda e^{t}+1}\right)^{r}e^{xt}\;\;\varphi(y,t)  \\
  &=&  \sum_{n=0}^{\infty}{}_{p}F_{n}^{^{(r)}}(x,y;\lambda,\mu,\nu)\frac{t^{n}}{n!},
\end{eqnarray*}
Thus, equating the coefficients of $t^{n}$ on both sides, we get\\
  ${}_{p}\mathbb{M}_{n}^{^{(r)}}(x,y;k,1;-\lambda)={}_{p}F_{n}^{^{(r)}}(x,y;\lambda,\mu,\nu)$\; where $\mu=1-k,\;\nu=k$.\\
( The 2-variable Apostol type polynomials of order $\alpha,$).

  \item setting $\alpha_{i}=\lambda,\;b=e,\;a=1\;and\; k=1  \; in\; \eqref{13},$ we have
\begin{eqnarray*}
 \sum_{n=0}^{\infty}{}_{p}\mathbb{M}_{n}^{^{(r)}}(x,y;1,1,e;\lambda)\frac{t^{n}}{n!}&=& \frac{(-1)^{r}2^{0} t^{r}}{(\lambda e^{t}-1)^{r}}e^{xt}\;\;\varphi(y,t)  \\
&=& (-1)^{r}\left(\frac{t}{\lambda e^{t}-1}\right)^{r}e^{xt}\;\;\varphi(y,t)  \\
  &=& (-1)^{r} \sum_{n=0}^{\infty}{}_{p}B_{n}^{^{(r)}}(x,y;\lambda)\frac{t^{n}}{n!},
\end{eqnarray*}
Thus, equating the coefficients of $t^{n}$ on both sides, we get\\
  ${}_{p}\mathbb{M}_{n}^{^{(r)}}(x,y;1,1,e;\lambda)=(-1)^{r}{}_{p}B_{n}^{^{(r)}}(x,y;\lambda)$.\\
( The 2-variable Apostol – Bernoulli polynomials of order $r$).

  \item setting $\alpha_{i}=-\lambda,\;b=e,\;a=1 \; and\; k=0\; in \;\eqref{13},$ we have
\begin{eqnarray*}
 \sum_{n=0}^{\infty}{}_{p}\mathbb{M}_{n}^{^{(r)}}(x,y;0,1,e;\lambda)\frac{t^{n}}{n!}&=& \frac{(-1)^{r}2^{r} t^{0}}{(-\lambda e^{t}-1)^{r}}e^{xt}\;\;\varphi(y,t)  \\
&=&\left(\frac{2}{\lambda e^{t}+1}\right)^{r}e^{xt}\;\;\varphi(y,t)  \\
  &=&  \sum_{n=0}^{\infty}{}_{p}E_{n}^{^{(r)}}(x,y;\lambda)\frac{t^{n}}{n!},
\end{eqnarray*}
Thus, equating the coefficients of $t^{n}$ on both sides, we get\\
  ${}_{p}\mathbb{M}_{n}^{^{(r)}}(x,y;0,1,e;\lambda)={}_{p}E_{n}^{^{(r)}}(x,y;\lambda)$.\\
( The 2-variable Apostol - Euler polynomials of order $r$).
  \item setting $\alpha_{i}=-\lambda,\;b=e,\;a=1 \;and\;k=1 in \eqref{13},$ we have
\begin{eqnarray*}
 \sum_{n=0}^{\infty}{}_{p}\mathbb{M}_{n}^{^{(r)}}(x,y;1,1,e;\lambda)\frac{t^{n}}{n!}&=& \frac{(-1)^{r}2^{0} t^{r}}{(-\lambda e^{t}-1)^{r}}e^{xt}\;\;\varphi(y,t)  \\
&=& (2)^{-r}\left(\frac{2t}{\lambda e^{t}+1}\right)^{r}e^{xt}\;\;\varphi(y,t)  \\
  &=& (2)^{-r} \sum_{n=0}^{\infty}{}_{p}G_{n}^{^{(r)}}(x,y;\lambda)\frac{t^{n}}{n!},
\end{eqnarray*}
Thus, equating the coefficients of $t^{n}$ on both sides, we get\\
  ${}_{p}\mathbb{M}_{n}^{^{(r)}}(x,y;1,1,e;-\lambda)=(2)^{-r}{}_{p}G_{n}^{^{(r)}}(x,y;\lambda)$.\\
( The 2-variable Apostol - Genocchi polynomials of order $r$).

\end{enumerate}

Now, we obtain the series definition of  ${}_{p}\mathbb{M}_{n}^{^{(r)}}(x,y;k,a,b;\bar\alpha_{r})$ by the following theorem:
\begin{theorem}
The 2-variable unified family of generalized Apostol-Euler, Bernoulli and Genocchi polynomials of order $(r)$ ${}_{p}\mathbb{M}_{n}^{^{(r)}}(x,y;k,a,b;\bar\alpha_{r})$ is defined by the series:
\begin{equation}\label{25}
  {}_{p}\mathbb{M}_{n}^{^{(r)}}(x,y;k,a,b;\bar\alpha_{r})=\sum_{k=0}^{n} \binom {n}{k}\mathbb{M}_{n-k}^{(r)}(a,b,\bar\alpha_{r})p_{k}(x,y).
\end{equation}
\end{theorem}
\begin{proof}
Using Eq.\eqref{n}, we have
\begin{eqnarray*}
 \sum_{n=0}^{\infty}{}_{p}\mathbb{M}_{n}^{^{(r)}}(x,y;k,a,b;\bar\alpha_{r})\frac{t^{n}}{n!}  &=&\frac{(-1)^{r}t^{rk}2^{r(1-k)}}{\prod\limits^{r-1}_{i=0}(\alpha_{i}b^{t}-a^{t})} \,\,\sum_{n=0}^{\infty}p_{n}(x,y)\frac{t^{n}}{n!}  \\
   &=& \sum_{n=0}^{\infty}\mathbb{M}_{n}^{^{(r)}}(a,b;\bar\alpha_{r})\frac{t^{k}}{k!}   \,\,\sum_{n=0}^{\infty}p_{n}(x,y)\frac{t^{n}}{n!}.
\end{eqnarray*}
By using Cauchy-product rule, we find
\begin{equation*}
 \sum_{n=0}^{\infty}{}_{p}\mathbb{M}_{n}^{^{(r)}}(x,y;k,a,b;\bar\alpha_{r})\frac{t^{n}}{n!}=\sum_{n=0}^{\infty}\sum_{k=0}^{n}\frac{\mathbb{M}_{n-k}^{^{(r)}}(a,b;\bar\alpha_{r})}{(n-k)!k!}p_{n}(x,y)t^{n}.
\end{equation*}
Equating the coefficients of the same powers of $t$ ,  yields \eqref{25}.
\end{proof}
\section{Implicit Summation Formulae for The 2-Variable Unified Family of Generalized Apostol-Euler, Bernoulli and Genocchi polynomials}
 \begin{theorem}
Let $a,\; b > 0$ and $a\neq b$. Then for $x, y, z \in \textbf{R}$ and $n \geq 0$.
The following implicit summation formula for ${}_{p}\mathbb{M}_{n}^{^{(r)}}(x,y;k,a,b;\bar\alpha_{r}) $  holds true:
\begin{equation}\label{26q}
{}_{p}\mathbb{M}_{n}^{^{(r)}}(x+z,y;k,a,b;\bar\alpha_{r})=\sum_{m=0}^{n}{}_{p}\mathbb{M}_{n}^{^{(r)}}(x,y;k,a,b;\bar\alpha_{r})z^{n-m}.
\end{equation}
\end{theorem}
\begin{proof}
Replacement of $x$ by $x+z$ in generating function \eqref{13} gives
\begin{eqnarray*}
  \sum_{n=0}^{\infty}{}_{p}\mathbb{M}_{n}^{^{(r)}}(x+z,y;k,a,b;\bar\alpha_{r})\frac{t^{n}}{n!}&=&
\frac{(-1)^{r}t^{rk}2^{r(1-k)}}{\prod\limits^{r-1}_{i=0}(\alpha_{i}b^{t}-a^{t})}e^{(x+z)t}\varphi(y,t)\\
&=&\frac{(-1)^{r}t^{rk}2^{r(1-k)}}{\prod\limits^{r-1}_{i=0}(\alpha_{i}b^{t}-a^{t})}e^{xt}\varphi(y,t) e^{zt}\\
&=&\sum_{m=0}^{\infty}{}_{p}\mathbb{M}_{n}^{(r)}(x,y;k,a,b;\bar\alpha_{r})\frac{t^{m}}{m!}\; \sum_{n=0}^{\infty}z^{n}\frac{t^{n}}{n!}\\
&=&\sum_{n=0}^{\infty}\sum_{m=0}^{\infty}{}_{p}\mathbb{M}_{n}^{^{(r)}}(x,y;k,a,b;\bar\alpha_{r}) z^{n}\frac{t^{n+m}}{n!\;m!}.
\end{eqnarray*}
which on replacing $n$ by $n-m$ in the r.h.s. and then equating the coefficients of the same powers of $t$ in both sides of the
last equation yields \eqref{26q}.
\end{proof}
\begin{theorem}
The following implicit summation formula for ${}_{p}\mathbb{M}_{n}^{^{(r)}}(x,y;k,a,b;\bar\alpha_{r}) $ in terms of generalized Apostol type polynomial $\mathbb{M}_{n}^{^{(r)}}(x;k,a,b;\bar\alpha_{r})$ is obtained:
\begin{equation}\label{27}
{}_{p}\mathbb{M}_{n}^{^{(r)}}(x+z,y;k,a,b;\bar\alpha_{r})=\sum_{k=0}^{n}\binom{n}{k}\mathbb{M}_{n}^{^{(r)}}(x;k,a,b;\bar\alpha_{r})p_{n-k}(z,y).
\end{equation}
\end{theorem}
\begin{proof}
Replacement of $x$ by $x+z$ in generating function \eqref{13}, we have
\begin{eqnarray*}
  \sum_{n=0}^{\infty}{}_{p}\mathbb{M}_{n}^{^{(r)}}(x+z,y;k,a,b;\bar\alpha_{r})\frac{t^{n}}{n!}&=&
\frac{(-1)^{r}t^{rk}2^{r(1-k)}}{\prod\limits^{r-1}_{i=0}(\alpha_{i}b^{t}-a^{t})}e^{(x+z)t}\varphi(y,t)\\
&=&\frac{(-1)^{r}t^{rk}2^{r(1-k)}}{\prod\limits^{r-1}_{i=0}(\alpha_{i}b^{t}-a^{t})}e^{xt}\varphi(y,t) e^{zt}\\
&=&\sum_{n=0}^{\infty}\mathbb{M}_{n}^{(r)}(x;k,a,b;\bar\alpha_{r})\frac{t^{n}}{n!}\; \sum_{k=0}^{\infty}p_{k}(z,y)\frac{t^{k}}{k!}.
\end{eqnarray*}
By applying Cauchy-product rule, we get
\begin{equation*}
   \sum_{n=0}^{\infty}{}_{p}\mathbb{M}_{n}^{^{(r)}}(x+z,y;k,a,b;\bar\alpha_{r})\frac{t^{n}}{n!}=
\sum_{n=0}^{\infty}\sum_{k=0}^{n}\binom{n}{k}\mathbb{M}_{n}^{^{(r)}}(x;k,a,b;\bar\alpha_{r})\;p_{n-k}(z,y) \frac{t^{n}}{n!}.
\end{equation*}
Equating the coefficients of $t^{n} $ on both sides, yields \eqref{27}.
\end{proof}

\begin{theorem}
Let $a,\; b > 0$ and $a\neq b$. Then for $x, y, z \in \textbf{R}$ and $n \geq 0$.
The following implicit summation formula for ${}_{p}\mathbb{M}_{n}^{^{(r)}}(x,y;k,a,b;\bar\alpha_{r}) $  holds true:
\begin{equation}\label{26}
{}_{p}\mathbb{M}_{n+m}^{^{(r)}}(z,y;k,a,b;\bar\alpha_{r})=\sum_{p,q=0}^{n,m}\binom{n}{p}\binom{m}{q}(z-x)^{p+q}{}_{p}\mathbb{M}_{n+m-p-q}^{^{(r)}}(x,y;k,a,b;\bar\alpha_{r}).
\end{equation}
\end{theorem}

\begin{proof}
we replace $t$ by $t+u$ in the generating function \eqref{13} and using the following rule \cite{sri}
\begin{equation}\label{sr}
\sum_{N=0}^{\infty} f(N)\frac{(x+y)^{N}}{N!}=\sum_{n,m=0}^{\infty}f(m+n)\frac{x^{n}}{n!}\frac{y^{m}}{m!},
\end{equation}
in the left hand side becomes
\begin{equation}\label{28}
 \sum_{n,m=0}^{\infty}{}_{p}\mathbb{M}_{n+m}^{^{(r)}}(x,y;k,a,b;\bar\alpha_{r})\frac{t^{n}}{n!}\frac{u^{m}}{m!}=
\frac{(-1)^{r}(t+u)^{rk}2^{r(1-k)}}{\prod\limits^{r-1}_{i=0}(\alpha_{i}b^{t+u}-a^{t+u})} \,\,e^{x(t+u)}\varphi (y,t+u),
 \end{equation}
Rewriting Eq. \eqref{28} as:
\begin{equation*}
 e^{-x(t+u)}\sum_{n,m=0}^{\infty}{}_{p}\mathbb{M}_{n+m}^{^{(r)}}(x,y;k,a,b;\bar\alpha_{r})\frac{t^{n}}{n!}\frac{u^{m}}{m!}=
\frac{(-1)^{r}(t+u)^{rk}2^{r(1-k)}}{\prod\limits^{r-1}_{i=0}(\alpha_{i}b^{t+u}-a^{t+u})} \,\,\varphi (y,t+u).
 \end{equation*}
Replacing $x$ by $z$ in the above equation and equating the resulting equation
to the above equation, we get
\begin{equation}\label{29}
 e^{(z-x)(t+u)}\sum_{n,m=0}^{\infty}{}_{p}\mathbb{M}_{n+m}^{^{(r)}}(x,y;k,a,b;\bar\alpha_{r})\frac{t^{n}}{n!}\frac{u^{m}}{m!}=
\sum_{n,m=0}^{\infty}{}_{p}\mathbb{M}_{n+m}^{^{(r)}}(z,y;k,a,b;\bar\alpha_{r})\frac{t^{n}}{n!}\frac{u^{m}}{m!}.
 \end{equation}
On expanding exponential function \eqref{29} gives
\begin{equation}\label{30}
 \sum_{N=0}^{\infty}\frac{[(z-x)(t+u)]^{N}}{N!}\sum_{n,m=0}^{\infty}{}_{p}\mathbb{M}_{n+m}^{^{(r)}}(x,y;k,a,b;\bar\alpha_{r})\frac{t^{n}}{n!}\frac{u^{m}}{m!}=
\sum_{n,m=0}^{\infty}{}_{p}\mathbb{M}_{n+m}^{^{(r)}}(z,y;k,a,b;\bar\alpha_{r})\frac{t^{n}}{n!}\frac{u^{m}}{m!},
 \end{equation}
using Eq. \eqref{sr} in the l.h.s. of Eq. \eqref{30}, we find
\begin{equation}\label{31}
 \sum_{p,q=0}^{\infty}\frac{(z-x)^{p+q}\;t^{p}\;u^{q}}{p!\;q!}\sum_{n,m=0}^{\infty}{}_{p}\mathbb{M}_{n+m}^{^{(r)}}(x,y;k,a,b;\bar\alpha_{r})\frac{t^{n}}{n!}\frac{u^{m}}{m!}=
\sum_{n,m=0}^{\infty}{}_{p}\mathbb{M}_{n+m}^{^{(r)}}(z,y;k,a,b;\bar\alpha_{r})\frac{t^{n}}{n!}\frac{u^{m}}{m!}.
 \end{equation}
Now replacing $n$ by $n-p$, $m$ by $m-q$ and using Cauchy-product rule in the left hand side of \eqref{31}, we get
\begin{equation}\label{32}
 \sum_{n,m=0}^{\infty}\sum_{p,q=0}^{n,m}\frac{(z-x)^{p+q}}{p!\;q!}{}_{p}\mathbb{M}_{n+m-p-q}^{^{(r)}}(x,y;k,a,b;\bar\alpha_{r})\frac{t^{n}}{(n-p)!}\frac{u^{m}}{(m-q)!}=
\sum_{n,m=0}^{\infty}{}_{p}\mathbb{M}_{n+m}^{^{(r)}}(z,y;k,a,b;\bar\alpha_{r})\frac{t^{n}}{n!}\frac{u^{m}}{m!}.
 \end{equation}
Finally, on equating the coefficients of the like powers of t and u in the above
equation, yields\eqref{26}.
\end{proof}

\begin{theorem}
The following implicit summation formula for ${}_{p}\mathbb{M}_{n}^{^{(r)}}(x,y;k,a,b;\bar\alpha_{r}) $  holds true:
\begin{equation}\label{33}
{}_{p}\mathbb{M}_{n}^{^{(r)}}(x+1,y;k,a,b;\bar\alpha_{r})=\sum_{m=0}^{n}\binom{n}{m}{}_{p}\mathbb{M}_{n-m}^{^{(r)}}(x,y;k,a,b;\bar\alpha_{r}).
\end{equation}
\end{theorem}

\begin{proof}
From Eq. \eqref{13}, we have
\begin{eqnarray*}
 \sum_{n=0}^{\infty} {}_{p}\mathbb{M}_{n}^{^{(r)}}(x+1,y;k,a,b;\bar\alpha_{r})\frac{t^{n}}{n!} &=& \frac{(-1)^{r}t^{rk}2^{r(1-k)}}{\prod\limits^{r-1}_{i=0}(\alpha_{i}b^{t}-a^{t})} \,\,e^{(x+1)t}\varphi (y,t), \\
   &=&\frac{(-1)^{r}t^{rk}2^{r(1-k)}}{\prod\limits^{r-1}_{i=0}(\alpha_{i}b^{t}-a^{t})} \,\,e^{xt}\varphi (y,t)\;e^{t}, \\
   &=& \sum_{n=0}^{\infty} {}_{p}\mathbb{M}_{n}^{^{(r)}}(x,y;k,a,b;\bar\alpha_{r})\frac{t^{n}}{n!}\sum_{m=0}^{\infty}\frac{t^{m}}{m!}.
\end{eqnarray*}
By using Cauchy-product rule, then
\begin{equation*}
   \sum_{n=0}^{\infty} {}_{p}\mathbb{M}_{n}^{^{(r)}}(x+1,y;k,a,b;\bar\alpha_{r})\frac{t^{n}}{n!} = \sum_{n=0}^{\infty}\sum_{m=0}^{n} \binom{n}{m} {}_{p}\mathbb{M}_{n-m}^{^{(r)}}(x,y;k,a,b;\bar\alpha_{r})\frac{t^{n}}{n!}.
\end{equation*}
Equating the coefficients of $t^{n} $ on both sides, yields \eqref{33}.
\end{proof}

\begin{theorem}
The following implicit summation formula for ${}_{p}\mathbb{M}_{n}^{^{(r)}}(x,y;k,a,b;\bar\alpha_{r}) $  holds true:
\begin{equation}\label{34}
{}_{p}\mathbb{M}_{n}^{^{(r)}}(x+z,y;k,a,b;\bar\alpha_{r})=\sum_{m=0}^{n}\binom{n}{m} \mathbb{M}_{n-m}^{^{(r)}}(z;k,a,b;\bar\alpha_{r})p_{m}(x,y).
\end{equation}
\end{theorem}

\begin{proof}
From Eq. \eqref{13}, we have
\begin{eqnarray*}
 \sum_{n=0}^{\infty} {}_{p}\mathbb{M}_{n}^{^{(r)}}(x+z,y;k,a,b;\bar\alpha_{r})\frac{t^{n}}{n!} &=& \frac{(-1)^{r}t^{rk}2^{r(1-k)}}{\prod\limits^{r-1}_{i=0}(\alpha_{i}b^{t}-a^{t})} \,\,e^{(x+z)t}\varphi (y,t), \\
   &=&\frac{(-1)^{r}t^{rk}2^{r(1-k)}}{\prod\limits^{r-1}_{i=0}(\alpha_{i}b^{t}-a^{t})} \,\,e^{zt} \; e^{xt}\varphi (y,t), \\
   &=& \sum_{n=0}^{\infty} \mathbb{M}_{n}^{^{(r)}}(z;k,a,b;\bar\alpha_{r})\frac{t^{n}}{n!}\sum_{m=0}^{\infty} p_{m}(x,y)\frac{t^{m}}{m!}.
\end{eqnarray*}
By using Cauchy-product rule, then
\begin{equation*}
   \sum_{n=0}^{\infty} {}_{p}\mathbb{M}_{n}^{^{(r)}}(x+z,y;k,a,b;\bar\alpha_{r})\frac{t^{n}}{n!} = \sum_{n=0}^{\infty}\sum_{m=0}^{n} \binom{n}{m} \mathbb{M}_{n-m}^{^{(r)}}(z;k,a,b;\bar\alpha_{r})\;p_{m}(x,y)\frac{t^{n}}{n!}.
\end{equation*}
Equating the coefficients of $t^{n} $ on both sides, yields \eqref{34}.
\end{proof}

\section{Symmetry identity}
\begin{theorem}\hfill\\
Let $c,d >0$ and $n\geq 0.$ For $x,y\in \mathbb{R},$ then the following identity holds true:
\begin{equation*}
 \sum_{m=0}^{n}\binom{n}{m} d^{m}c^{n-m}\; {}_{p}\mathbb{M}_{n-m}^{(r)}(dx,y;k,a,b;\bar{\alpha_{r}})\; {}_{p}\mathbb{M}_{m}^{(r)}(cx,y;k,a,b;\bar{\alpha_{r}})
 \end{equation*}
 \begin{equation}\label{35}
 = \sum_{m=0}^{n}\binom{n}{m}c^{m} d^{n-m} \; {}_{p}\mathbb{M}_{n-m}^{(r)}(cx,y;k,a,b;\bar{\alpha_{r}})\; {}_{p}\mathbb{M}_{m}^{(r)}(dx,y;k,a,b;\bar{\alpha_{r}}).
\end{equation}
\end{theorem}
\begin{proof}
Let
\begin{equation*}
  G(t)=\frac{((-1)^{r}2^{r(1-k)}t^{rk})^{2}}{(\prod\limits_{i=0}^{r-1}(\alpha_{i}b^{dt}-a^{dt}))
  (\prod\limits_{i=0}^{r-1}(\alpha_{i}b^{ct}-a^{ct}))}\; e^{2cdxt}\varphi(y,dt)\varphi(y,ct).
\end{equation*}
Then the expression for $G(t)$ is symmetric in $c$ and $d$ and we can expand $G(t)$ into series in two ways\\
Firstly
\begin{eqnarray*}
  G(t) &=& \frac{1}{(cd)^{rk}}\left(\frac{(-1)^{r}2^{r(1-k)}(dt)^{rk}}{\prod\limits_{i=0}^{r-1}(\alpha_{i}b^{dt}-a^{dt})}\right)\; e^{cxdt}\left(\frac{(-1)^{r}2^{r(1-k)}(ct)^{rk}}{\prod\limits_{i=0}^{r-1}(\alpha_{i}b^{ct}-a^{ct})}\right)\; e^{dxct}\\
 &=& \frac{1}{(ab)^{rk}}\sum_{n=0}^{\infty} {}_{p}\mathbb{M}_{n}^{(r)}(cx,y;k,a,b;\bar{\alpha_{r}})\frac{(dt)^{n}}{n!}\sum_{m=0}^{\infty} {}_{p}\mathbb{M}_{n}^{(r)}(dx,y;k,a,b;\bar{\alpha_{r}})\frac{(ct)^{m}}{m!}\\
\end{eqnarray*}
\begin{equation}\label{36}
G(t)= \frac{1}{(ab)^{rk}}\sum_{n=0}^{\infty}\sum_{m=0}^{n} \binom{n}{m} d^{n-m}\;c^{m}\;{}_{p}\mathbb{M}_{n-m}^{(r)}(cx,y;k,a,b;\bar{\alpha_{r}})\;{}_{p}\mathbb{M}_{m}^{(r)}(dx,y;k,a,b;\bar{\alpha_{r}})\frac{(t)^{n}}{n!}.
\end{equation}
Secondly
\begin{equation}\label{37}
  G(t)=\frac{1}{(cd)^{rk}}\sum_{n=0}^{\infty}\sum_{m=0}^{n} \binom{n}{m} c^{n-m}\;d^{m}\;{}_{p}\mathbb{M}_{n-m}^{(r)}(dx,y;k,a,b;\bar{\alpha_{r}})\;{}_{p}\mathbb{M}_{n}^{(r)}(cx,y;k,a,b;\bar{\alpha_{r}})\frac{(t)^{n}}{n!}.
\end{equation}
Form Eq. \eqref{36} and Eq. \eqref{37}, by comparing the coefficients of $t^{n}$ on the both sides, yields \eqref{35}.
\end{proof}

\section{Examples}
 By making suitable choice for the function $\varphi(y, t)$ in equation \eqref{13}, the generating function for the corresponding member belonging to the  ${}_{p}\mathbb{M}_{n}^{^{(r)}}(x,y;k,a,b;\bar\alpha_{r})$ family can be obtained.\\ \\
\textbf{Example 1.}
Taking $\varphi(y, t) = e^{yt^{m}}$ (for which the $p_{n}(x, y)$ reduce to the  $H_{n} ^{(m)}(x, y))$ in the l.h.s. of generating function \eqref{13}, we find that the resultant Gould-Hopper Apostol type polynomials (GHATP), denoted by ${}_{H^{(m)}}\mathbb{M}_{n}^{^{(r)}}(x,y;k,a,b;\bar\alpha_{r})$
in the r.h.s. are defined by the following generating function:
\begin{equation}\label{38}
\sum_{n=0}^{\infty}{}_{H^{(m)}}\mathbb{M}_{n}^{^{(r)}}(x,y;k,a,b;\bar\alpha_{r})\frac{t^{n}}{n!}=\frac{(-1)^{r}t^{rk}2^{r(1-k)}}{\prod\limits^{r-1}_{i=0}(\alpha_{i}b^{t}-a^{t})} \,\,e^{xt+yt^{m}}.
\end{equation}
\begin{remark}
For $m = 2$, the $H_{n} ^{(m)}(x, y)$ reduce to the $H_{n}(x, y)$. Therefore, taking $m = 2$ in Eq. \eqref{38},
we get the following generating function for the 2-variable Hermite Apostol type polynomials, denoted by
${}_{H}\mathbb{M}_{n}^{^{(r)}}(x,y;k,a,b;\bar\alpha_{r})$:
\end{remark}
\begin{equation}\label{39}
\sum_{n=0}^{\infty}{}_{H}\mathbb{M}_{n}^{^{(r)}}(x,y;k,a,b;\bar\alpha_{r})\frac{t^{n}}{n!}=\frac{(-1)^{r}t^{rk}2^{r(1-k)}}{\prod\limits^{r-1}_{i=0}(\alpha_{i}b^{t}-a^{t})} \,\,e^{xt+yt^{2}}.
\end{equation}
\textbf{Example 2.}
Taking $\varphi(y, t) = C_{0}(-yt^{m})$ (for which the $p_{n}(x, y)$ reduce to the  ${}_{m}L_{n}(y,x))$ in the l.h.s. of generating function \eqref{13}, we find that the resultant 2-variable generalized Laguerre Apostol type polynomials, denoted by ${}_{mL}\mathbb{M}_{n}^{^{(r)}}(y,x;k,a,b;\bar\alpha_{r})$ in the r.h.s. are defined by the following generating function:
\begin{equation}\label{40}
\sum_{n=0}^{\infty}{}_{mL}\mathbb{M}_{n}^{^{(r)}}(y,x;k,a,b;\bar\alpha_{r})\frac{t^{n}}{n!}=\frac{(-1)^{r}t^{rk}2^{r(1-k)}}{\prod\limits^{r-1}_{i=0}(\alpha_{i}b^{t}-a^{t})} \,\,e^{xt}C_{0}(-yt^{m}).
\end{equation}
\begin{remark}
Since for $m = 1$ and $y\longrightarrow -y$, the ${}_{m}L_{n}(x, y)$ reduce to the $L_{n}(x, y)$. Therefore, taking $m = 1$  and $y\longrightarrow -y$ in Eq. \eqref{40},
we get the following generating function for the 2-variable Laguerre Apostol type polynomials, denoted by ${}_{L}\mathbb{M}_{n}^{^{(r)}}(y,x;k,a,b;\bar\alpha_{r})$:
\end{remark}
\begin{equation}\label{41}
\sum_{n=0}^{\infty}{}_{L}\mathbb{M}_{n}^{^{(r)}}(y,x;k,a,b;\bar\alpha_{r})\frac{t^{n}}{n!}=\frac{(-1)^{r}t^{rk}2^{r(1-k)}}{\prod\limits^{r-1}_{i=0}(\alpha_{i}b^{t}-a^{t})} \,\,e^{xt}C_{0}(yt).
\end{equation}

\begin{remark}
Since for $x = 1$ , the $L_{n}(x, y)$ reduce to the classical Laguerre polynomials $L_{n}(y)$. Therefore, taking $x = 1$ in Eq. \eqref{41},
we get the following generating function for the Laguerre Apostol type polynomials, denoted by ${}_{L}\mathbb{M}_{n}^{^{(r)}}(y;k,a,b;\bar\alpha_{r})$:
\end{remark}
\begin{equation}\label{42}
\sum_{n=0}^{\infty}{}_{L}\mathbb{M}_{n}^{^{(r)}}(y;k,a,b;\bar\alpha_{r})\frac{t^{n}}{n!}=\frac{(-1)^{r}t^{rk}2^{r(1-k)}}{\prod\limits^{r-1}_{i=0}(\alpha_{i}b^{t}-a^{t})} \,\,e^{t}C_{0}(yt).
\end{equation}

\textbf{Example 3.}
Taking $\varphi(y, t) = \frac{1}{1-yt^{\beta}}$ (for which the $p_{n}(x, y)$ reduce to the $e_{n}^{(\beta)}(x,y))$ in the l.h.s. of generating function \eqref{13}, we find that the resultant 2-variable truncated exponential Apostol type polynomials of order $\beta$, denoted by ${}_{e^{^{(\beta)}}}\mathbb{M}_{n}^{^{(r)}}(x,y;k,a,b;\bar\alpha_{r})$ in the r.h.s. are defined by the following generating function:
\begin{equation}\label{43}
\sum_{n=0}^{\infty}{}_{e^{^{(\beta)}}}\mathbb{M}_{n}^{^{(r)}}(x,y;k,a,b;\bar\alpha_{r})\frac{t^{n}}{n!}=\frac{(-1)^{r}t^{rk}2^{r(1-k)}}{\prod\limits^{r-1}_{i=0}(\alpha_{i}b^{t}-a^{t})} \,\,\left(\frac{e^{xt}}{1-yt^{\beta}}\right).
\end{equation}

\begin{remark}
Since for $\beta = 2$, the $e_{n}^{(\beta)}(x, y)$ of order $\beta$ reduce to the $ {}_{[2]}e_{n}(x, y)$. Therefore, taking $\beta= 2$ in
Eq. \eqref{43}, we get the following generating function for the 2-variable truncated exponential Apostol type polynomials
, denoted by ${}_{{}_{[2]}e}\mathbb{M}_{n}^{^{(r)}}(x,y;k,a,b;\bar\alpha_{r})$
\end{remark}
\begin{equation}\label{44}
\sum_{n=0}^{\infty}{}_{{}_{[2]}e}\mathbb{M}_{n}^{^{(r)}}(x,y;k,a,b;\bar\alpha_{r})\frac{t^{n}}{n!}=\frac{(-1)^{r}t^{rk}2^{r(1-k)}}{\prod\limits^{r-1}_{i=0}(\alpha_{i}b^{t}-a^{t})} \,\,\left(\frac{e^{xt}}{1-yt^{2}}\right).
\end{equation}

\begin{remark}
Since for $y = 1$, the ${}_{[2]}e_{n}(x, y)$ of order $\beta$ reduce to the truncated exponential polynomials $ {}_{[2]}e_{n}(x)$. Therefore, taking $y= 1$ in
Eq. \eqref{44}, we get the following generating function for the truncated exponential polynomials
, denoted by ${}_{{}_{[2]}e}\mathbb{M}_{n}^{^{(r)}}(x;k,a,b;\bar\alpha_{r})$
\end{remark}
\begin{equation}\label{45}
\sum_{n=0}^{\infty}{}_{{}_{[2]}e}\mathbb{M}_{n}^{^{(r)}}(x;k,a,b;\bar\alpha_{r})\frac{t^{n}}{n!}=\frac{(-1)^{r}t^{rk}2^{r(1-k)}}{\prod\limits^{r-1}_{i=0}(\alpha_{i}b^{t}-a^{t})} \,\,\left(\frac{e^{xt}}{1-t^{2}}\right).
\end{equation}

\bibliographystyle{plain}


\end{document}